\documentclass[11pt,letterpaper]{amsart}
\usepackage[utf8]{inputenc}
\usepackage{float}
\usepackage{amsmath}
\usepackage{amsfonts}
\usepackage{amssymb}
\usepackage{amsthm}
\usepackage{graphicx}
\usepackage{lmodern}
\usepackage{color}
\usepackage[left=2cm,right=2cm,top=2cm,bottom=2cm]{geometry}

\newtheorem{theorem}{Theorem}[section]

\newtheorem{proposition}[theorem]{Proposition}
\newtheorem{prop}[theorem]{Proposition}

\newcommand{\C}{\mathbb{C}}
\newcommand{\R}{\mathbb{R}}
\newcommand{\RP}{\mathbb{RP}}

\newcommand{\Z}{\mathbb{Z}}

\newcommand{\Pa}{{\mathcal{P}}}

\begin{document}

\title{Apollonian Packings and Kac-Moody Root Systems}

\author{Ian Whitehead}
\date{}
\begin{abstract}
We study Apollonian circle packings in relation to a certain rank 4 indefinite Kac-Moody root system $\Phi$. We introduce the generating function $Z(\mathbf{s})$ of a packing, an exponential series in four variables with an Apollonian symmetry group, which relates to Weyl-Kac characters of $\Phi$. By exploiting the presence of affine and Lorentzian hyperbolic root subsystems of $\Phi$, with automorphic Weyl denominators, we express $Z(\mathbf{s})$ in terms of Jacobi theta functions and the Siegel modular form $\Delta_5$.  We also show that the domain of convergence of $Z(\mathbf{s})$ is the Tits cone of $\Phi$, and discover that this domain inherits the intricate geometric structure of Apollonian packings.
\end{abstract}
\maketitle

\section{Introduction} \label{intro}

The aim of this article is to study Apollonian circle packings from the perspective of Kac-Moody theory, motivating new questions about packings and Kac-Moody root systems. First we must explain why these two topics are related at all. Figure \ref{quadruple} shows a quadruple of mutually tangent circles in the plane, with curvatures $c_1, c_2, c_3, c_4$. 

\begin{figure}[h]
\center{\includegraphics[scale=.5]{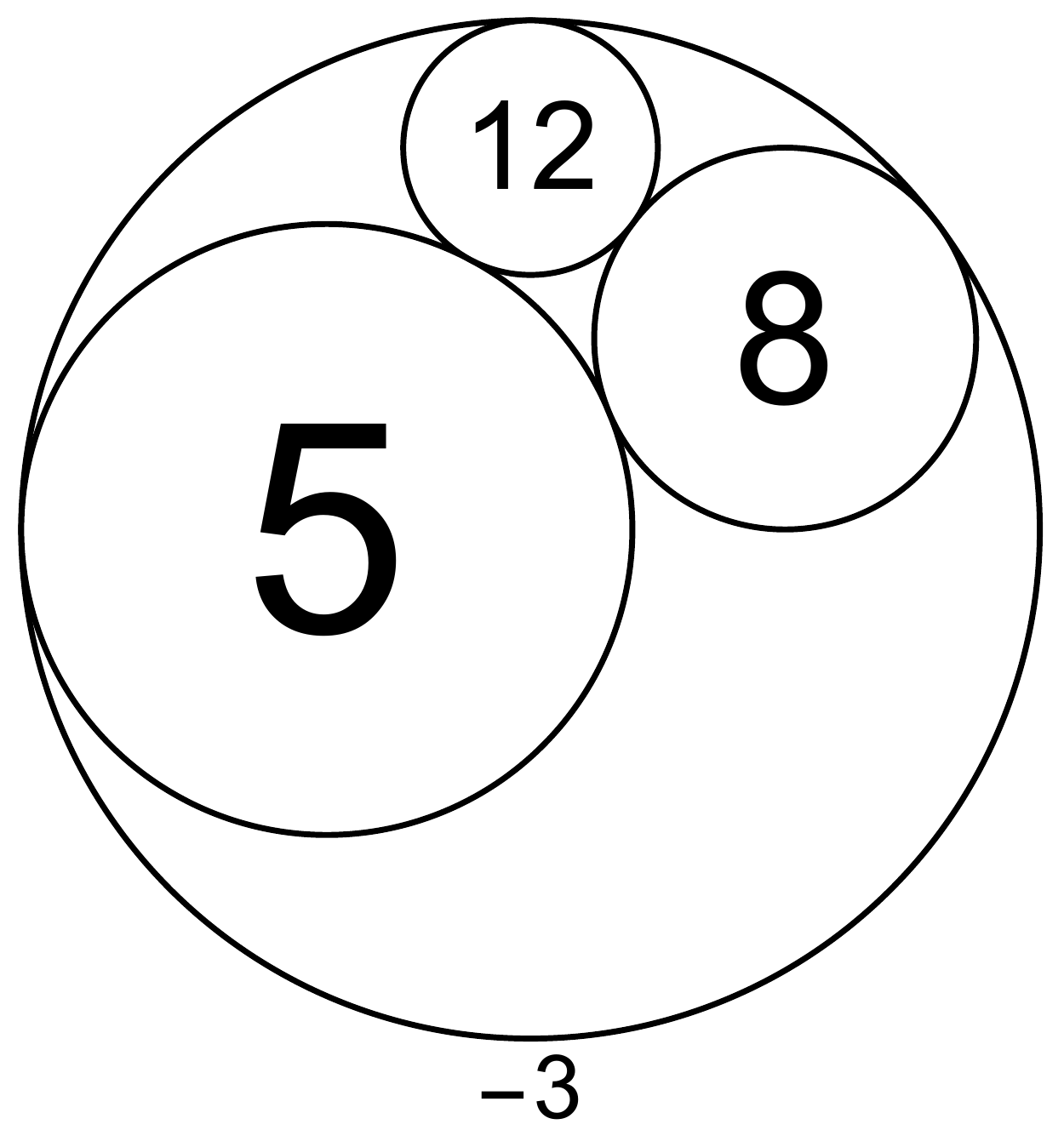}}
\caption{Descartes quadruple with $-3\alpha_1+5\alpha_2+8\alpha_3+12\alpha_4$}
\label{quadruple}
\end{figure}

By convention, we take $c_i$ to be negative if this circle is external to the other three; a circle can also degenerate to a straight line with curvature $0$. Descartes discovered that the four curvatures satisfy a quadratic equation:
\begin{equation}\label{DescartesForm}
2c_1^2+2c_2^2+2c_3^2+2c_4^2-(c_1+c_2+c_3+c_4)^2 = 0
\end{equation}
The quadratic form on the left side of \eqref{DescartesForm} corresponds to the Cartan matrix:
\begin{equation}
\left(\begin{array}{rrrr} 2 & -2 & -2 & -2 \\ -2 & 2 & -2 & -2 \\ -2 & -2 & 2 & -2 \\ -2 & -2 & -2 & 2 \end{array}\right)
\end{equation}
of an indefinite symmetric Kac-Moody root system $\Phi$. To get a sense of the complexity of $\Phi$, note that the principal submatrices of rank 2 yield root subsystems of affine type $A_1^{(1)}$. The principal submatrices of rank 3 yield root subsystems of hyperbolic type $H^{(3)}_{71}$. (The notation $H^{(3)}_{71}$ is taken from Carbone et al. \cite{MR2608277}; this root system is called $A_{1, II}$ in \cite{MR1438983} and $\Pi_{3, 1}$ in \cite{MR3335126}.) However, the root system $\Phi$ itself is not affine or hyperbolic.

The Cartan matrix of $\Phi$ defines a symmetric bilinear form $(\, , \, )$ of signature $(3,1)$ on $\R^4$ equipped with the basis of simple real roots $\alpha_1, \alpha_2, \alpha_3, \alpha_4$. The Weyl group of $\Phi$ is the Apollonian group: 
\begin{equation}
W=\langle \sigma_1, \sigma_2, \sigma_3, \sigma_4 | \sigma_i^2=1 \rangle
\end{equation}
with action on $\R^4$ determined by $\sigma_i(\alpha_j)=\alpha_j-(\alpha_j, \alpha_i)\alpha_i$. This preserves the form $F$. If $c_1, c_2, c_3, c_4$ are the curvatures of a quadruple of mutually tangent circles, the action of $W$ on $\mathbf{c}=c_1\alpha_1+c_2\alpha_2+c_3\alpha_3+c_4\alpha_4$ has a beautiful geometric interpretation. If three mutually tangent circles are fixed, then there exist exactly two circles which are tangent to all three. The mapping $\sigma_i$ corresponds to fixing circles of curvature $c_j$ for $j\neq i$, and swapping out the circle of curvature $c_i$. This can be interpreted as a M\"{o}bius transformation of the complex plane: an inversion across a circle containing the points of tangency of the three fixed circles. 

The orbit of $W$ on an initial Descartes quadruple $\mathbf{c}$ is an Apollonian circle packing, shown in Figure \ref{packing}.

\begin{figure}[h]
\center{\includegraphics[scale=.75]{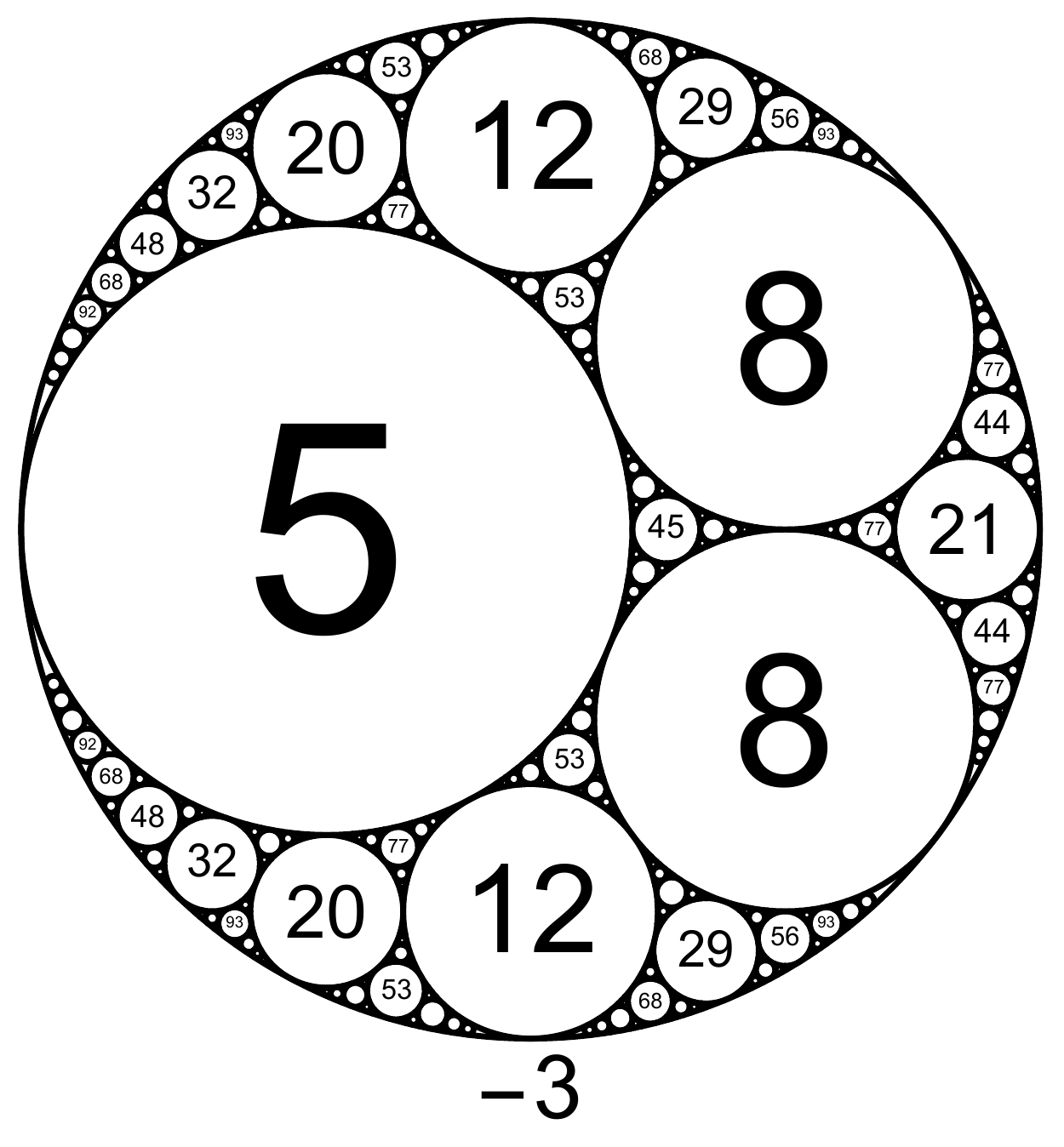}}
\caption{Apollonian packing}
\label{packing}
\end{figure}

A quadruple of mutually tangent circles $\mathbf{c}' = c_1'\alpha_1+c_2'\alpha_2+c_3'\alpha_3+c_4'\alpha_4$ appears in this figure if and only if it can be obtained from $\mathbf{c}$ by an element of $W$. We will denote the multiset of Descartes quadruples in the packing as $\Pa$ (the multiplicity of quadruples in the packing will be discussed in the proof of Prop. \ref{prop:firstdomain} below). Notice that if the initial quadruple  $\mathbf{c} \in \Z\alpha_1 \oplus \Z\alpha_2 \oplus \Z\alpha_3 \oplus \Z\alpha_4$, then so is every quadruple in the packing; in this case the packing is called integral. The number-theoretic study of integral Apollonian packings has experienced a renaissance in the last 20 years; see \cite{MR1971245, MR2800340}.

A bounded integral packing $\Pa$ can be considered as a $W$-orbit in the root lattice of $\Phi$. This orbit exhibits behavior not seen in finite, affine, or even hyperbolic types: it is bounded below but not contained in the positive cone. We define the height $\mathrm{ht}(c_1\alpha_1+c_2\alpha_2+c_3\alpha_3+c_4\alpha_4)=c_1+c_2+c_3+c_4$. There exists a unique quadruple in $\Pa$, called the base quadruple, of minimal height. (This is usually called the root quadruple, but we use the term base quadruple to avoid confusion with root system terminology.) The base quadruple is an antidominant vector in the root lattice. The orbit $\Pa$ consists of vectors $\mathbf{c}$ of length squared $(\mathbf{c},\mathbf{c})=0$, and height bounded below. Imaginary roots in the root lattice also have length squared $\leq 0$ and height bounded below. But imaginary roots have another important property which elements of $\Pa$ lack: all their $W$-translates lie in either the positive or negative cone. The base quadruple in $\Pa$, and others involving the exterior circle, have three positive coordinates and one negative. Thus $\Pa$ is not an orbit of imaginary roots. The unbounded integral packings with base quadruple some permutation of $(n,n,0,0)$ do correspond to orbits of imaginary roots; indeed, these are imaginary roots of the affine root subsystems $A_1^{(1)}$.  

Fix a bounded packing $\Pa$ with base quadruple $\mathbf{c}$. We will study a series that can be considered as a generating function for $\Pa$, or a symmetric function for the Apollonian group $W$. Let $\omega_1, \omega_2, \omega_3, \omega_4$ denote the dual basis of fundamental weights in $\Phi$. For $s_1, s_2, s_3, s_4 \in \C$, let $\mathbf{s}=s_1 \omega_1 +  s_2 \omega_2 + s_3 \omega_3 + s_4 \omega_4$. Define
\begin{equation}
Z(\mathbf{s})=\sum_{w \in W} e^{-(w\mathbf{c}, \mathbf{s})}=\sum_{c_1\alpha_1+c_2 \alpha_2+c_3\alpha_3+c_4\alpha_4 \in \Pa} e^{-c_1s_1-c_2s_2-c_3s_3-c_4s_4}
\end{equation}
This series inherits an infinite group of symmetries $W$ from the symmetries of $\Pa$. Its analytic properties--convergence, growth, zeroes and poles--can translate into information about the asymptotic behavior of quadruples in $\Pa$.  

We will study two features of the series $Z(\mathbf{s})$: its relation to automorphic forms and its convergence. In Section \ref{theta}, we relate it to Jacobi theta functions using the affine $A_1^{(1)}$ root subsystem. Theorem \ref{prop:theta} gives an expansion of $Z(\mathbf{s})$ in terms of theta functions. In Section \ref{delta}, we take advantage of the hyperbolic $H^{(3)}_{71}$ root subsystem, which is a foundational example in Gritsenko and Nikulin's theory of Lorentzian Kac-Moody root systems \cite{MR1616925, MR1616929}. This theory yields a surprising connection between $Z(\mathbf{s})$ and the Siegel automorphic form $\Delta_5$ on $\mathrm{Sp}(4)$. Theorem \ref{theorem:Delta} makes this connection. These sections are intended to lay the groundwork for further study of $Z(\mathbf{s})$ from an automorphic perspective. 

In sections \ref{convergence} and \ref{geometry} we describe the domain of absolute convergence of $Z(\mathbf{s})$, a four-dimensional region which we call the Apollonian cone $A$. Theorem \ref{theorem:A} establishes that the Apollonian cone is essentially the Tits cone for $\Phi$. This domain is independent of $\Pa$ and has a rich geometry related to Apollonian packings. In Theorem \ref{theorem:Ageometry}, we give a complete geometric description of $A$. A cross-section of $A$ is shown in Figure \ref{domaininfinity}. This domain is reminiscent of the fundamental domain for the Apollonian group acting on a 3-dimensional hyperbolic half-space, as in \cite{MR2784325} and elsewhere. However, our approach does not use the isomorphism between $\mathrm{SO}^+(3,1)$ and $\mathrm{PSL}(2, \C)$, or any hyperbolic geometry at all. It is purely on the orthogonal group side. In essence, these sections allow us to rediscover Apollonian packings based solely on the Descartes quadratic form and the associated root system. Our argument provides a template to study the geometry of Tits cones for Kac-Moody root systems more generally.

Chen and Labb\'{e} have studied the set of limit roots in certain Kac-Moody root systems and related them to sphere packings \cite{MR3303040}. Their work involves similar visualizations to ours of the action of hyperbolic Coxeter groups on the root space. However, the Apollonian cone constructed here seems to be original, although the method of construction requires no specialized tools. 

We will sketch one application of the series $Z(\mathbf{s})$, which also illustrates why automorphicity is an important concern. There has been great interest in the Apollonian ``$L$-function''
\begin{equation}
L(u)= \sum_{c \in \Pa^*} c^{-u}
\end{equation}
Here $\Pa^*$ is the collection of curvatures of circles in $\Pa$, again counted with multiplicity). This series is known to converge for $\Re(u) > \delta$, where $\delta\approx 1.30568$ is the Hausdorff dimension of the residual set of any packing \cite{MR493763}. Meromorphic continuation to the left of $\delta$ would yield an asymptotic for the growth of circles in $\Pa$. Important work of Kontorovich and Oh \cite{MR2784325} and of Lee and Oh \cite{MR3053757} has shown that
\begin{equation}
|\lbrace c \in \Pa^* | c<X \rbrace | =r X^{\delta} + O(X^{\delta-\tfrac{2(\delta-s_1)}{63}})
\end{equation}
where $r$ is a constant depending on the packing, and $s_1$ is a constant independent of the packing. Their approach is based on equidistribution of horocycles on a hyperbolic 3-manifold and does not explicitly involve $L(u)$. Meromorphic continuation of $L(u)$ would yield a new proof.

In fact, $L(u)$ can be obtained from $Z(\mathbf{s})$ by an integral transform. First, for $t>0$, we take
\begin{equation} \label{mellin}
\begin{split}
Z_1(t)=\tfrac{1}{2\pi i} \int_{(\tfrac{1}{2})} \, & Z(st\omega_1+st\omega_2+st\omega_3+(1-s)t\omega_4)+Z(st\omega_1+st\omega_2+(1-s)t\omega_3+st\omega_4) \\ & +Z(st\omega_1+(1-s)t\omega_2+st\omega_3+st\omega_4)+Z((1-s)t\omega_1+st\omega_2+st\omega_3+st\omega_4) \, \tfrac{ds}{s}
\end{split}
\end{equation}
where the integral on the vertical line $\Re(s)=\tfrac{1}{2}$ is taken in the principal value sense.  

An individual summand in $Z(st\omega_1+st\omega_2+st\omega_3+(1-s)t\omega_4)$ has the form $e^{(-c_1-c_2-c_3+c_4)st-c_4t}$. The integral in $s$ will be $0$ if $c_1+c_2+c_3 > c_4$, $\tfrac{e^{-c_4t}}{2}$ if $c_1+c_2+c_3 = c_4$, and $e^{-c_4t}$ if $c_1+c_2+c_3 < c_4$. The results of the integration for the three other terms of the integrand are entirely parallel. For any Descartes quadruple $(c_1, c_2, c_3, c_4) \in \Pa$ other than the base quadruple, there is a unique Apollonian group reflection $\sigma_i$ which reduces $c_i$, yielding a Descartes quadruple of larger circles in $\Pa$. For this $i$, we have $c_i>\sum\limits_{j \neq i} c_j$, and $c_i$ is the maximal circle in the quadruple. It follows that the integral of 
\begin{equation}
e^{(-c_1-c_2-c_3+c_4)st-c_4t}+e^{(-c_1-c_2+c_3-c_4)st-c_3t}+e^{(-c_1+c_2-c_3-c_4)st-c_2t}+e^{(c_1-c_2-c_3-c_4)st-c_1t}
\end{equation}
will be simply $e^{-\max(c_i) t}$. The base quadruple does not contribute to the integral at all. Note that packings with symmetry type $D_2$, i.e. with the base quadruple being a multiple of $(-1, 2, 2, 3)$, contain two copies of the base quadruple. In this case each copy of the base quadruple $(c_1, c_2, c_3, c_4)$ will contribute $\tfrac {e^{-\max(c_i) t}}{2}$, so it is as if one copy of the base quadruple is removed. 

Because each $c \in \Pa^*$ is the maximum of a unique Descartes quadruple other than the base quadruple, we have shown that 
\begin{equation}
Z_1(t)+e^{-c_1t}+e^{-c_2t}+e^{-c_3t}+e^{-c_4t}=\sum_{c \in \Pa^*} e^{-ct}
\end{equation}
where $c_1, c_2, c_3, c_4$ are the four curvatures of the base quadruple. Finally, a Mellin transform in $t$ yields
\begin{equation} \label{mellin2}
\int_0^{\infty} \, t^u (Z_1(t)+e^{-c_1t}+e^{-c_2t}+e^{-c_3t}+e^{-c_4t}) \, \tfrac{dt}{t} \, = \, \Gamma(u) L(u) 
\end{equation}
We will see from Proposition \ref{prop:firstdomain} that $Z(\mathbf{s})$ converges absolutely in the domain of the two integrations. The first integral converges conditionally in the principal value sense. The integrand in the second integral has a potential singularity at $t=0$. The integral converges as $t \to \infty$ because of the rapid decay of $Z_1(t)$, but it may diverge as $t \to 0$ for sufficiently small $\Re(u)$. 

One would hope to meromorphically continue $L(u)$ following the procedure of Riemann's second proof of the meromorphic continuation and functional equation for the zeta function. This requires finding a symmetry for $Z_1(t)$ in $t \mapsto \tfrac{1}{t}$. Such a symmetry does not arise from the group of functional equations $W$ for $Z(\mathbf{s})$, but it might come from additional automorphic behavior. In particular, both the theta functions of Section \ref{theta} and the Siegel automorphic form $\Delta_5$ of Section \ref{delta} possess such symmetries. 

The problem is one of automorphic correction. Characters of affine Kac-Moody root systems are essentially theta functions with an $\mathrm{SL}(2, \, \Z)$ symmetry beyond the affine Weyl group \cite{Kac}. Beginning with examples due to Feingold and Frenkel \cite{MR697333} and  Borcherds \cite{MR1069386}, there have been automorphic interpretations for the Weyl denominator functions of certain indefinite Kac-Moody root systems. The principle seems to be that these functions are not automorphic on their own, but, in many cases, one can add extra imaginary roots to produce a generalized Kac-Moody root system whose denominator has automorphic properties. In its sum expression, the corrected denominator function contains infinitely many different Weyl orbits, not just one. In its product expression, the multiplicities of the imaginary roots can be expressed explicitly. 

Gritsenko and Nikulin give conditions on generalized Kac-Moody algebras which are good candidates for automorphic correction in \cite{MR1992083}. These are called Lorentzian Kac-Moody algebras (not to be confused with Lorentzian inner forms and lattices). They have two fundamental properties. The first, the existence of a lattice Weyl vector $\rho$, is straightforward to check for $\Phi$. The second property, that a fundamental domain for $W$ has finite hyperbolic volume in the cone of vectors with negative length squared, does not hold for $\Phi$. This will be discussed further in Section \ref{geometry}.

So the Apollonian root system $\Phi$ itself is not one of the best candidates for automorphic correction. Perhaps this does not foreclose the possibility of generalizing Riemann's proof as discussed above, but it indicates the inherent difficulty of doing so. As another indication, note that the constant $\delta$ would appear in the calculation as a pole of $L(u)$. Since little is known about this constant, it is not clear how it would arise. A more tractable problem is to make use of other Lorentzian root subsystems of $\Phi$ (beyond the principal ones). This can yield new information about the density of curvatures appearing in different subsets of an Apollonian packing. 

Having established a connection between Apollonian packings and Kac-Moody root systems, we will suggest some possible generalizations on both sides. On the Kac-Moody side, one could begin with an indefinite quadratic form of similar complexity to $F$, and ask what fractal geometry arises from its Weyl group of symmetry and Tits cone. Is some analogue of an Apollonian packing involved? What kinds of discrete group actions on hyperbolic spaces are obtained in this way? Can a fractal dimension analogous to $\delta$ be associated to other Kac-Moody algebras? One could also take the Mellin transform for a character or symmetric function associated to a Weyl orbit on the root lattice, as in equation \eqref{mellin2}. What kinds of $L$-series result, and what are their convergence properties?

The Apollonian group is an important example of a ``thin group,'' with orbits that are dense but of infinite covolume in the ambient space \cite{MR3020826}. An important problem in harmonic analysis is to understand how much of the theory of automorphic forms extends to this context. One could ask which Cartan matrices yield Weyl groups that are thin in their orthogonal groups. This condition is in tension with Gritsenko and Nikulin's finite volume condition. The answer would yield an interesting new class of thin reflection groups, and a new class of Kac-Moody algebras beyond the Lorentzian ones.

On the packing side, there are many possible generalizations to consider: higher dimensional packings like the sphere packings of Boyd \cite{MR350626} and Maxwell \cite{MR679972}, Apollonian superpackings \cite{MR2183489}, the octahedral packing of Guettler and Mallows \cite{MR2675919}, and more. Kontorovich and Nakamura introduce a classification of crystallographic sphere packings in all dimensions \cite{MR3904690}. They give notions of integrality and superintegrality for general crystallographic sphere packings, and show that the latter yields a finite classification. Stange introduces a collection of packings associated to imaginary quadratic fields and Bianchi groups \cite{MR3814328}. In all these cases, the basic unit of the packing is a tuple of circles whose curvatures satisfy one or more quadratic forms. The group of symmetries is generated by reflections which preserve the forms. One could ask which generalizations of Apollonian packings are related to a Kac-Moody root system, and which root systems arise this way. Do any especially interesting root systems, e.g. Lorentzian ones, appear? A generating function like $Z(\mathbf{s})$ can be associated to any of these packings; how does the geometry of its domain of convergence relate to the geometry of the packing itself? 

\subsection{Acknowledgements}

The author thanks Alex Kontorovich, Kate Stange, Holley Friedlander, Cathy Hsu, Anna Pusk\'{a}s, Dinakar Muthiah, and Li Fan for helpful conversations that shaped this project. This work is dedicated, with gratitude, to Joel Carr, David Gomprecht, and Michael Sturm. 

\section{Expansion of $Z(\mathbf{s})$ in Theta Functions} \label{theta}

The Kac-Moody root system $\Phi$ has a principal rank 2 root subsystem $A_1^{(1)}$ with the Cartan matrix $\left(\begin{array}{cc} 2 & -2 \\ -2 & 2 \end{array}\right)$. The Weyl group of this root subsystem is an infinite dihedral group. Sums over Weyl orbits are theta functions--this is equivalent to the fact that the set of circles tangent to two fixed circles in a packing have curvatures parametrized by a quadratic polynomial. Theta functions satisfy a group $\mathrm{GL}(2, \Z)$ of symmetries, in which the Weyl group elements act as upper-triangular matrices. Their appearance here is preliminary evidence that $Z(\mathbf{s})$ may have automorphic properties beyond its Apollonian group of symmetries. In this section we will briefly explain the connection between $Z(\mathbf{s})$ and theta functions.

Let $W_2$ denote the subgroup $\langle \sigma_3, \sigma_4 \rangle \subset W$, which is the Weyl group of an $A_1^{(1)}$ root subsystem. Fix a pair of tangent circles in $\Pa$, assuming without loss of generality that their curvatures are $c_1, c_2$ in a Descartes quadruple $ \mathbf{c}=c_1 \alpha_1+c_2\alpha_2+c_3\alpha_3+c_4\alpha_4$. The set of all Descartes quadruples including these two circles is an orbit of $W_2$ in $\Pa$. Define
\begin{equation}
Z_2^+(\mathbf{s})=\sum_{w \in W_2} e^{-(w\mathbf{c}, \mathbf{s})}, \qquad Z_2^-(\mathbf{s})=\sum_{w \in W_2} (-1)^{\ell(w)} e^{-(w\mathbf{c}, \mathbf{s})}
\end{equation}
where $\ell(w)$ denotes the length of a reduced word for $w$. The following proposition relates these to Jacobi theta functions, which we denote as
\begin{equation}
\theta_{00}(z, \tau)=\sum_{n \in \Z} e^{2\pi i n z +\pi i n^2 \tau}, \qquad \theta_{01}(z, \tau)=\sum_{n \in \Z} (-1)^n e^{2\pi i n z +\pi i n^2 \tau}
\end{equation}

\begin{prop}\label{prop:theta}
We have:
\begin{equation}
\begin{split}
& Z_2^{\pm}(\mathbf{s}) = e^{-(\mathbf{c}, \mathbf{s})}  \left(\frac{\theta_{00}+\theta_{01}}{2}\right) \left( \frac{(c_1+c_2+c_3-c_4)s_3-(c_1+c_2-c_3+c_4)s_4}{2 \pi i}, \frac{ -(c_1+c_2)(s_3+s_4)}{\pi i} \right) \\
&  \pm e^{-(\mathbf{c}, \mathbf{s})+(c_3-c_4)(s_3-s_4)}\left(\frac{\theta_{00}-\theta_{01}}{2}\right) \left( \frac{-(c_1+c_2-c_3+c_4)s_3+(c_1+c_2+c_3-c_4)s_4}{2 \pi i}, \frac{ -(c_1+c_2)(s_3+s_4)}{\pi i} \right) 
\end{split}
\end{equation}
\end{prop}
\begin{proof}
The set of Descartes quadruples in $\Pa$ containing $c_1$, $c_2$ may be parametrized as follows:
\begin{equation}
\begin{split}
& \{ \mathbf{c} + n((n-1)c_1+(n-1)c_2- c_3+ c_4)\alpha_3 + n((n+1)c_1+(n+1)c_2- c_3+ c_4)\alpha_4 \, | \, n \in \Z \text{ even} \}  \\
& \cup \{ \mathbf{c} +  (n+1)(n c_1+ n c_2- c_3+ c_4) \alpha_3 + (n-1)(n c_1+ n c_2-c_3+ c_4)\alpha_4 \, | \, n \in \Z \text{ odd} \}
\end{split}
\end{equation}
where the first subset arises from applying words of even length in $\sigma_3, \sigma_4$ to $(c_1, c_2, c_3, c_4)$, and the second arises from applying words of odd length. We can then write
\begin{equation}
\begin{split}
&Z_2^{\pm}(\mathbf{s}) = e^{-(\mathbf{c}, \mathbf{s})} \left( \sum_{n \in \Z \text{ even}} e^{-n^2(c_1+c_2)(s_3+s_4)+n(c_1+c_2+c_3-c_4)s_3-n(c_1+c_2-c_3+c_4)s_4}\right. \\
& \pm \left. \sum_{n \in \Z \text{ odd}} e^{-n^2(c_1+c_2)(s_3+s_4)-n(c_1+c_2-c_3+c_4)s_3+n(c_1+c_2+c_3-c_4)s_4+(c_3-c_4)(s_3-s_4)}\right)
\end{split}
\end{equation}
which is equivalent to the desired formula. 
\end{proof}

It follows from this proposition that we may write:
\begin{equation}
\begin{split}
Z(\mathbf{s}) = \sum_{c_1, c_2} & e^{-(\mathbf{c}, -\mathbf{s})}\left(\tfrac{\theta_{00}+\theta_{01}}{2}\right)\left(\tfrac{(c_1+c_2)(s_3-s_4)+(c_3-c_4)(s_3+s_4)}{2\pi i}, \tfrac{-(c_1+c_2)(s_3+s_4)}{\pi i}\right)   \\
& +e^{-(\mathbf{c}, \mathbf{s})+(c_3-c_4)(s_3-s_4)} \left(\tfrac{\theta_{00}-\theta_{01}}{2}\right) \left(\tfrac{(c_1+c_2)(s_4-s_3)+(c_3-c_4)(s_3+s_4)}{2\pi i}, \tfrac{-(c_1+c_2)(s_3+s_4)}{\pi i}\right)\,
\end{split}
\end{equation}
where the sum is over pairs $c_1, c_2$ such that some Descartes quadruple $\mathbf{c}=c_1\alpha_1+c_2\alpha_2+c_3\alpha_3+c_4\alpha_4$ appears in $\Pa$. It does not matter which quadruple $\mathbf{c}$ we choose to associate to $c_1, c_2$.

In the special case $c_3=c_4$, which can occur if the packing $\Pa$ has a line of symmetry, $Z_2^+(\mathbf{s})$ and $Z_2^-(\mathbf{s})$ behave especially nicely. In this case, $Z_2^-(\mathbf{s})$ is essentially the Weyl denominator for $A_1^{(1)}$ rather than a general alternating sum over the Weyl group. We have 
\begin{equation}
\begin{split}
 &Z_2^+(\mathbf{s}) = e^{-(\mathbf{c}, \mathbf{s})} \theta_{00} \left( \frac{(c_1+c_2)(s_3-s_4)}{2 \pi i}, \frac{ -(c_1+c_2)(s_3+s_4)}{\pi i} \right), \\
& Z_2^-(\mathbf{s}) = e^{-(\mathbf{c}, \mathbf{s})} \theta_{01} \left( \frac{(c_1+c_2)(s_3-s_4)}{2 \pi i}, \frac{ -(c_1+c_2)(s_3+s_4)}{\pi i} \right) 
\end{split}
\end{equation}
The series $Z_2^+(\mathbf{s})$ and $Z_2^-(\mathbf{s})$ admit Jacobi triple product expressions and satisfy simpler transformation laws with respect to $\mathrm{GL}(2, \Z)$. A related simplification occurs with specialized values of $\mathbf{s}$, such as those appearing in the integral transform of equation \eqref{mellin}.

\section{Relation to the Siegel Modular form $\Delta_5$}\label{delta}

The Kac-Moody root system $\Phi$ has a principal rank 3 root subsystem $H_{71}^{(3)}$ with the Cartan matrix
\begin{equation}
\left(\begin{array}{ccc} 2 & -2 & -2 \\ -2 & 2 & -2 \\ -2 & -2 & 2 \end{array}\right)
\end{equation}
This root system is hyperbolic and Lorentzian. Indeed, it is one of the original examples of a Lorentzian root system, studied by Gritsenko and Nikulin in \cite{MR1428063}. They furnish an automorphic correction of this root system whose Weyl denominator is the Siegel automorphic form $\Delta_5$ on $\mathrm{Sp}(4)$. In this section we will outline the relationship between $Z(\mathbf{s})$ and $\Delta_5$. 

An orbit of the Weyl group of $H_{71}^{(3)}$ in $\Pa$ is simply the collection of Descartes quadruples including a fixed circle. This collection plays an important role in the literature on Apollonian packings. After a change of variables, the Weyl group is isomorphic to the congruence subgroup $\Gamma_0(2)$ of $\mathrm{GL}(2, \Z)$; its action on Descartes quadruples is isomorphic to the action of $\Gamma_0(2)$ on binary quadratic forms. As a consequence, one can show that the curvatures of circles tangent to a fixed circle in a packing $\Pa$ are precisely the values taken by a shifted binary quadratic form. This has been a crucial tool for proving density results on the family of curvatures--see \cite{MR2813334}. 

From this change of variables, we can see directly that a sum over the Weyl group of $H_{71}^{(3)}$ has $\mathrm{GL}(2, \Z)$ symmetries. The surprising fact is that such a sum may possess a larger group $\mathrm{Sp}(4, \Z)$ of symmetries, in which $\mathrm{GL}(2, \Z)$ is the subgroup of block diagonal matrices. This fact has not been used in the literature. Because of the technical details of automorphic correction, we will encounter some obstacles in applying this $\mathrm{Sp}(4)$ automorphicity to density and counting problems in packings. But new results may be possible, especially if we broaden to consider other Lorentzian root subsystems of $\Phi$. 

Let $W_3$ denote the subgroup $\langle \sigma_2, \sigma_3, \sigma_4 \rangle \subset W$. which is the Weyl group of an $H_{71}^{(3)}$ root subsystem. Fix a Descartes quadruple $\mathbf{c}=c_1\alpha_1+c_2\alpha_2+c_3\alpha_3+c_4\alpha_4 \in \Pa$. We assume from the start that $c_2=c_3=c_4$ and, by rescaling if necessary, that $2c_1+2c_2=1$. This ensures that a sum over the Weyl orbit of $\mathbf{c}$ behaves like the Weyl denominator for $H_{71}^{(3)}$. These assumptions cannot be satisfied in an integral plane packing, but they can with the non-integral packing with $D_3$ symmetry whose base quadruple is $-\frac{\sqrt{3}}{4}\alpha_1+\frac{1}{4} \left(2+\sqrt{3}\right)\alpha_2+\frac{1}{4} \left(2+\sqrt{3}\right)\alpha_3+\frac{1}{4} \left(2+\sqrt{3}\right)\alpha_4$. They can also be satisfied in the integral spherical and hyperbolic packings studied in \cite{MR2357448}.

As in Section \ref{theta}, define
\begin{equation}
Z_3^+(\mathbf{s})=\sum_{w \in W_3} e^{-(w\mathbf{c}, \mathbf{s})}, \qquad Z_3^-(\mathbf{s})=\sum_{w \in W_3} (-1)^{\ell(w)} e^{-(w\mathbf{c}, \mathbf{s})}
\end{equation}
Our goal is to relate $Z_3^-(\mathbf{s})$ to the Siegel modular form $\Delta_5$.

Let us fix some notation. The group $\mathrm{Sp}(4, \Z)$ consists of integral $4 \times 4$ matrices $M=\begin{pmatrix} A & B \\ C & D \end{pmatrix}$ such that $\,^t M \begin{pmatrix} 0 & I \\ -I & 0 \end{pmatrix} M =  \begin{pmatrix} 0 & I \\ -I & 0 \end{pmatrix}$. Here $A$, $B$, $C$, $D$, $0$ and $I$ denote $2 \times 2$ block matrices. The Siegel upper half plane $\mathbb{H}_2$ is the set of symmetric $2 \times 2$ complex matrices $Z=X+iY$ such that the imaginary part $Y$ is a positive-definite matrix. $\mathrm{Sp}(4, \Z)$ acts on $\mathbb{H}_2$ via 
\begin{equation}
\begin{pmatrix} A & B \\ C & D \end{pmatrix} Z = (AZ+B)(CZ+D)^{-1}
\end{equation} 
A Siegel modular form $f$ of weight $k \in \Z$ and character $\nu:\mathrm{Sp}(4, \Z) \to \C^{\times}$ is a holomorphic function on $\mathbb{H}_2$ satisfying 
\begin{equation}
f (MZ) = \nu(M) \det (CZ+D)^k f(Z)
\end{equation} 
for all $M\in \mathrm{Sp}(4, \Z)$.

The function $\Delta_5: \mathbb{H}_2 \to \C$ is a Siegel cusp form of weight $5$ with a nontrivial quadratic character $\nu$. For full details on the construction of $\Delta_5$, we refer the reader to \cite{MR1428063}. Here we will work with the Fourier expansion of $\Delta_5$:
\begin{equation}
\frac{1}{64} \Delta_5(Z) = \sum_{\substack{l, m, n \text{ odd} \\ m, n, 4mn-l^2>0}} \sum_{d | \gcd(l, m, n)} d^4 g\left(\frac{mn}{d^2}. \frac{l}{d}\right) e^{\pi i(nz_1 + l z_2 + mz_3)}
\end{equation}
where $Z=\begin{pmatrix} z_1 & z_2 \\ z_2 & z_3 \end{pmatrix} \in \mathbb{H}_2$ \cite[Equation 4.10]{MR1428063}. The coefficients $g(k, l)$ are defined by the generating series 
\begin{equation}
\sum_{k, l \text{ odd}} g(k, l) e^{\pi i(kz_1 + l z_2)}  = \eta(z_1)^9\theta_{11}(z_2, z_1) 
\end{equation}
where $\eta(z)=e^{\pi i z/12}\prod\limits_{n=1}^{\infty} (1-e^{2 \pi i n z})$ and $\theta_{11}(z, \tau)= \sum\limits_{n \in \Z} (-1)^n e^{\pi i (2n+1)z+\pi i (n+1/2)^2 \tau}$. Then the Jacobi triple product formula yields:
\begin{equation}
\sum_{k, l \text{ odd}} g(k, l) e^{\pi i(kz_1 + l z_2)}  = -e^{\pi i (z_1-z_2)}\prod_{n=1}^{\infty}(1-e^{2 \pi i((n-1)z_1+z_2)})(1-e^{2 \pi i(nz_1-z_2)})(1-e^{2 \pi i nz_1})^{10}
\end{equation}

In order to relate the action of $\mathrm{Sp}(4, \Z)$ on $\mathbb{H}_2$ to the action of the Apollonian group, we introduce a new basis of $\R^4$ and change variables. Let $\beta_1 = (\alpha_3+\alpha_4)/2$, $\beta_2=\alpha_4$, $\beta_3=(\alpha_2+\alpha_4)/2$.  Then $\mathbf{s}=s_1\omega_1+s_2\omega_2+s_3\omega_3+s_4\omega_4$ can be rewritten as $z_0 \omega_1+ z_1 \beta_1 + z_2 \beta_2+ z_3 \beta_3$ for some $z_0, z_1, z_2, z_3 \in \C$. This is essentially the same change of variables used to relate the action of the Apollonian group to the action of $\Gamma_0(2)$ on binary quadratic forms. We also let $\rho=\tfrac{1}{2}\alpha_2+\tfrac{1}{2}\alpha_3 + \tfrac{1}{2} \alpha_4$, the Weyl vector for the $H_{71}^{(3)}$ root subsystem.

\begin{theorem}\label{theorem:Delta}
For $\mathbf{s}=s_1\omega_1+s_2\omega_2+s_3\omega_3+s_4\omega_4=z_0 \omega_1+ z_1 \beta_1 + z_2 \beta_2+ z_3 \beta_3$, we have
\begin{equation}\label{Deltaeq}
e^{(\mathbf{c}, \mathbf{s})} \left( Z_3^-(\mathbf{s}) -\sum_\alpha m(\alpha) \sum_{w \in W_3} (-1)^{\ell(w)} e^{-(w(\mathbf{c}+\alpha), \mathbf{s})} \right) = \frac{e^{(\rho, \mathbf{s})}}{64} \Delta_5\left( \frac{1}{\pi i}Z\right)
\end{equation}
where the first sum is over $\alpha \in \Z_{\scriptscriptstyle{\geq 0}} \alpha_2 \oplus \Z_{\scriptscriptstyle{\geq 0}} \alpha_3 \oplus \Z_{\scriptscriptstyle{\geq 0}} \alpha_4$ such that $(\alpha, \alpha_i) \leq 0$ for $i=2, 3, 4$, and the $m(\alpha)$ are integer constants. Further, 
\begin{equation}\label{Deltaeq2}
\frac{e^{(\rho, \mathbf{s})}\Delta_5\left( \tfrac{1}{\pi i}Z\right)}{64 e^{(\mathbf{c}, \mathbf{s})} Z_3^-(\mathbf{s})}
\end{equation}
is a series of exponentials of the form $e^{-(\beta, \mathbf{s})}$ where each $\beta$ is a nonnegative integer combination of $\alpha_2$, $\alpha_3$, $\alpha_4$ satisfying $(\beta, \beta) \leq 0$. 
\end{theorem}
The first statement is the analog of Theorem 2.3 in \cite{MR1428063}. The meaning of the second statement comes from comparing the product forms of the Weyl denominator for $H^{(3)}_{71}$, $e^{(\mathbf{c}, \mathbf{s})} Z_3^-(\mathbf{s})$, to the Weyl denominator of its automorphic correction, $\tfrac{e^{(\rho, \mathbf{s})}}{64} \Delta_5\left( \tfrac{1}{\pi i}Z\right)$. The original root system and the automorphic correction have the same real roots, and differ only by imaginary roots. 
\begin{proof}
Note that $(\mathbf{c}, \alpha_i)=(\rho, \alpha_i)=-1$ for $i=2, 3, 4$ because $c_2=c_3=c_4$ and $2c_1+2c_2=1$. Therefore, for $\alpha$ as in the proposition, the bilinear pairing of $\mathbf{c}+\alpha$ with any positive root in the root subsystem will be a nonpositive integer. That is, $\mathbf{c}+\alpha$ behaves like an antidominant weight. Further, for $w \in W_3$, we have $\mathbf{c}-w(\mathbf{c}+\alpha) = \rho-w(\rho+\alpha)$, and this is a nonpositive integer combination of $\alpha_2, \alpha_3, \alpha_4$. Applying the bilinear form, we see that $(\mathbf{c}, \mathbf{s})-(w(\mathbf{c}+\alpha), \mathbf{s})=(\rho, \mathbf{s})-(w(\rho+\alpha), \mathbf{s})$ is an even integer combination of $z_1, z_2, z_3$, with nonnegative coefficients of $z_1$ and $z_3$. Since $(\rho, \mathbf{s})=-z_1-z_2-z_3$, we can write:
\begin{equation}
-(w(\rho+\alpha), \mathbf{s}) = nz_1 + l z_2 + mz_3
\end{equation}
With $n, l, m$ odd, and $n, m>0$. The condition that $4mn-l^2>0$ is equivalent to $(w(\rho+\alpha), w(\rho+\alpha))<0$, which holds because $\rho+\alpha$ is positive and antidominant and $w$ preserves the bilinear from. From this calculation and the Fourier expansion of $\Delta_5$, we see that the two sides of \eqref{Deltaeq} are exponential sums with the same support, and it suffices to compare the coefficients. 

By the Fourier expansion, the coefficients on the right side of \eqref{Deltaeq} are integers. Moreover, the constant coefficient, which corresponds to $l=m=n=1$, is $1$. We must show that the Fourier coefficients on the right side are alternating with respect to the action of $W_3$. This action is generated as follows: if $-(w(\rho+\alpha), \mathbf{s}) = nz_1 + l z_2 + mz_3$, then
\begin{equation}
\begin{split}
&-(\sigma_2w(\rho+\alpha), \mathbf{s}) = (n-2l+4m)z_1 + (4m-l) z_2 + mz_3 \\
&-(\sigma_3w(\rho+\alpha), \mathbf{s}) = nz_1 + (4n-l) z_2 + (m-2l+4n)z_3 \\
&-(\sigma_4w(\rho+\alpha), \mathbf{s}) = nz_1 - l z_2 + mz_3
\end{split}
\end{equation}
To show alternation for $\sigma_2$, we must check that
\begin{equation} 
\sum_{d | \gcd(4m-l, m, n-2l+4m)} d^4 g\left(\frac{m(n-2l+4m)}{d^2}. \frac{4m-l}{d}\right) = -\sum_{d | \gcd(l, m, n)} d^4 g\left(\frac{mn}{d^2}. \frac{l}{d}\right)
\end{equation}
Since the sum over $d$ is the same on both sides, it suffices to check that $g\left(m(n-2l+4m), 4m-l \right) = -g(mn, l)$, or more generally, that $g(k-2lm+4m^2, 4m-l) = -g(k, l)$ for all $k, l, m \in \Z$. We can see from the definition that this property holds for coefficients of $\theta_{11}(z_2, z_1)$, and therefore it must hold for the coefficients $g(k, l)$ of $\eta(z_1)^9\theta_{11}(z_2, z_1)$. The proof of alternation for $\sigma_3$ is entirely parallel. For $\sigma_4$, it suffices to show that $g(nm, -l)=-g(nm, l)$, which again is apparent from the series definition of $\theta_{11}(z_2, z_1)$. 

Now that we have the expected $W_3$ alternation on both sides of \eqref{Deltaeq}, note that each $W_3$ orbit contains a unique antidominant element, which can be written as $\rho+\alpha$ where $-(\rho+\alpha, \mathbf{s}) = nz_1 + l z_2 + mz_3$, and $\alpha$ is a nonnegative integer combination of $\alpha_2$, $\alpha_3$, $\alpha_4$. In this case, set 
\begin{equation}
m(\alpha)= -\sum\limits_{d | \gcd(l, m, n)} d^4 g\left(\frac{mn}{d^2}. \frac{l}{d}\right)
\end{equation} 
and equation \eqref{Deltaeq} follows.

To justify the final sentence of the theorem, use \eqref{Deltaeq} to express \eqref{Deltaeq2} as a linear combination of terms of the form 
\begin{equation}
\frac{\sum_{w \in W_3} (-1)^{\ell(w)} e^{(\mathbf{c}-w(\mathbf{c}+\alpha), \mathbf{s})}}{\sum_{w \in W_3} (-1)^{\ell(w)} e^{(\mathbf{c}-w(\mathbf{c}), \mathbf{s})}}
\end{equation}
Both the numerator and denominator are series of exponentials $e^{-(\beta, \mathbf{s})}$ where each $\beta$ is a nonnegative integer combination of $\alpha_2$, $\alpha_3$, $\alpha_4$. The denominator is a unit in the ring of such series. Finally, the quotient is $W_3$-invariant, so if it includes a term $e^{-(\beta, \mathbf{s})}$, then it also includes $e^{-(w(\beta), \mathbf{s})}$ for all $w \in W_3$. Thus every element in the $W_3$ orbit of $\beta$ is a nonnegative integer combination of $\alpha_2$, $\alpha_3$, $\alpha_4$. If $(\beta, \beta)>0$, then $(\beta, \alpha_i)>0$ for $i=1$, $2$, or $3$, and we can apply $\sigma_i$ to reduce the height of $\beta$. We can repeat this procedure until we have a negative coefficient of $\alpha_2$, $\alpha_3$, or $\alpha_4$, a contradiction. We conclude that $(\beta, \beta) \leq 0$.
\end{proof}

In the group $\mathrm{Sp}(4, \Z)$ of symmetries for $\Delta_5$, $W_3$ is embedded as the subgroup of block matrices $\begin{pmatrix} A & 0 \\ 0 & \,^t A^{-1} \end{pmatrix}$ with $A \in \Gamma_0(2) \subset \mathrm{GL}(2, \Z)$. The symmetry under $\begin{pmatrix} 0 & I \\ -I & 0 \end{pmatrix} \in \mathrm{Sp}(4, \Z)$ is of the kind needed to complete the Mellin inversion argument sketched in Section \ref{intro}. The difference between $e^{(\mathbf{c}, \mathbf{s})} Z_3^-(\mathbf{s})$ and its automorphic correction $\tfrac{e^{(\rho, \mathbf{s})}}{64} \Delta_5\left( \tfrac{1}{\pi i}Z\right)$ adds some technical complication, but this method can be used to estimate the density of curvatures tangent to a fixed circle in $\Pa$. We do not pursue this here because good estimates for the density of integers represented by a shifted binary quadratic form are already available \cite{MR2267284}. Another promising approach is to consider alternate Lorentzian root subsystems of $\Phi$.  

Throughout this section, we have assumed that $\mathbf{c}$ has a particularly simple form, which causes $Z_3^-(\mathbf{s})$ to behave like the Weyl denominator for the $H_{71}^{(3)}$ root system. A related simplification occurs with specializations of $\mathbf{s}$, like those appearing in \eqref{mellin}. It is not clear whether an arbitrary sum over an orbit of $W_3$ can be automorphically corrected. 

\section{The Apollonian Cone $A$} \label{convergence}

Since several different conelike objects appear below, we introduce some terminology here. We define the cone with apex $\mathbf{p}\in \R^n$ on a region $R \subset \R^n$ as the union of all rays originating at $\mathbf{p}$ and containing a point of $R$. The bounded cone with apex $\mathbf{p}$ on region $R$ is the union of all line segments between $\mathbf{p}$ and a point of $R$. A simplicial cone is the cone on a simplex or finite union of simplices in $\R^n$. 

We begin by asking where $Z(\mathbf{s})$ converges. By standard results in the theory of several complex variables, the domain of absolute convergence of any series of exponentials with real coefficients is a convex tube domain. The real parts of the variables $s_i$ must lie in a convex region in $\R^4$, while the imaginary parts can be arbitrary. The following proposition establishes an initial domain of absolute convergence.

\begin{proposition} \label{prop:firstdomain}
$Z(\mathbf{s})$ is absolutely convergent in the simplicial cone $C_0$ defined by the four inequalities $\Re(s_i)>0$.
\end{proposition}

\begin{proof} 
Let $c$ denote the negative curvature of the exterior circle in $\Pa$. Assume without loss of generality that the base quadruple is ordered from smallest to largest curvature. Then we have $c_1 \geq c$ and $c_2, \, c_3, \, c_4 \geq 0$ for all quadruples $c_1\alpha_1+c_2\alpha_2+c_3\alpha_3+c_4\alpha_4 \in \Pa$.

Any ordered quadruple appears for at most two Descartes configurations in $\Pa$. Indeed, if the same ordered quadruple of curvatures appears at distinct configurations, then the same sequence of moves in the Apollonian group can be applied to both quadruples to obtain two distinct copies of the base Descartes configuration in $\Pa$. This is possible if and only if the packing has symmetry type $D_2$. In this case there are exactly two copies of the ordered base quadruple in the packing, and thus two copies of any ordered quadruple. In any other case, there is only one copy of each ordered quadruple.

Therefore $Z(\mathbf{s})$ can be compared to the product of four geometric series 
\begin{equation}
\frac{2e^{-cs_1}}{(1-e^{-s_1})(1-e^{-s_2})(1-e^{-s_3})(1-e^{-s_4})}
\end{equation} 
which converges absolutely in this region.
\end{proof}



A similar argument, counting triples of circles in the packing instead of quadruples, can be used to prove the following proposition, which gives a larger domain of convergence.

\begin{proposition}
$Z(\mathbf{s})$ is absolutely convergent in the simplicial cone $C_1$ defined by the 12 inequalities $\Re(s_i) > - 2 \Re(s_j)$. 
\end{proposition}




In fact, the domain of convergence is even larger. The defining property of $Z(\mathbf{s})$ is its invariance under the Apollonian group $W=\langle \sigma_1, \sigma_2, \sigma_3, \sigma_4 \rangle$. Applying $w\in W$ to $\mathbf{s}$ simply permutes the summands of $Z(\mathbf{s})$, preserving absolute convergence. 

Applying $\sigma_i$ to the cone $C_0$ yields a simplicial cone defined by the four inequalities $\Re(s_i)<0$ and $\Re(s_j)>-2\Re(s_i)$ for $j \neq i$. From this we see that $C_1$ is the convex hull of the set $C_0 \cup \sigma_1(C_0) \cup  \sigma_2(C_0)  \cup \sigma_3(C_0) \cup  \sigma_4(C_0)$. In particular, this set includes the faces of $\bar{C_0}$ where exactly one of the $\Re(s_i)$ is $0$, but does not include the 2-skeleton of $\bar{C_0}$. Indeed, if the real parts of two of the $s_i$ are $0$, then $Z(\mathbf{s})$ certainly diverges because there exist infinitely many quadruples in the packing with two circles fixed.

Let $C=\lbrace \mathbf{s} \in \C^4 \, | \, \text{all } \Re(s_i) \geq 0, \text{ at most one } \Re(s_i)=0 \rbrace$. Then we have the following:
\begin{theorem} \label{theorem:A} 
The domain of absolute convergence of $Z(\mathbf{s})$ is $A=\bigcup_{w \in W} \, w(C)$.
\end{theorem}

This domain is closely related to the Tits cone,  $\bigcup_{w \in W} w(\bar{C})$. Notice also that this domain is independent of the Apollonian packing $\Pa$ that we have fixed. 

\begin{proof}
We may assume for simplicity that our point of convergence $\mathbf{s}$ lies in $\R^4$, with the understanding that the real domain of absolute convergence determines the complex domain. 

It is clear from the invariance of $Z(\mathbf{s})$ under $W$ and from the previous propositions that the domain of absolute convergence contains $A$. We must show that if $Z(\mathbf{s})$ converges absolutely at $\mathbf{s}$, then this point belongs to $A$. If any $W$-translate of $\mathbf{s}$ has two or more nonpositive coordinates, then $Z(\mathbf{s})$ diverges, because the series will contain infinitely many terms with absolute value bounded below by some positive constant. If some $W$-translate has all nonnegative coordinates and at most one zero coordinate, then the point lies in $A$ by definition. 

The only remaining possibility is that every $W$-translate of $\mathbf{s}$ has three positive coordinates and one negative. Let $B$ be the set of points with this property, and suppose $\mathbf{s} \in B$. Define the height of $\mathbf{s}$ as $s_1+s_2+s_3+s_4$. Note that applying $\sigma_i$ adds $4s_i$ to the height of a point. Every point in $B$ has positive height, because there is only one negative coordinate $s_i$, and two of the three other coordinates must exceed $-2s_i$. Moreover, for a point $\mathbf{s} \in B$, there is a unique sequence of $W$-translates:
\begin{equation}
\mathbf{s}, \, \sigma_{i_1}\mathbf{s}, \, \sigma_{i_2}\sigma_{i_1}\mathbf{s}, \, \sigma_{i_3}\sigma_{i_2}\sigma_{i_1}\mathbf{s}, \, \ldots
\end{equation}
with height decreasing monotonically. The heights of this sequence must converge to a lower bound $h$, so the negative coordinates of the sequence must converge to $0$. It follows that the sequence of points eventually lies in the compact region $h \leq s_1+s_2+s_3+s_4 \leq h+\epsilon$, $s_1, s_2, s_3, s_4 \geq -\epsilon$. We can choose a convergent subsequence, denoted $(\mathbf{s}_n)$. Because the negative coordinate $s_i$ of $\mathbf{s}_n$ converges to $0$, and one other coordinate $s_j$ is bounded above by $-2s_i$, the point $\lim_{n \to \infty} \, \mathbf{s}_n$ must lie on the 2-skeleton of the simplicial cone $\bar{C_0}$.

Note that the function $Z(\mathbf{s})$ on $\R^4$ is defined by a sum of positive terms. If it converges at $\mathbf{s}$, then it converges to the same value at each $\mathbf{s}_n$. This implies that the partial sums at each $\mathbf{s}_n$ are uniformly bounded above, which in turn implies that the partial sums at $\lim_{n \to \infty} \, \mathbf{s}_n$ are bounded above. We would then have that $Z(\mathbf{s})$ converges at $\lim_{n \to \infty} \, \mathbf{s}_n$, a contradiction.
\end{proof}

\section{The Geometry of $A$} \label{geometry}

What does the domain of absolute convergence $A$ look like? The question is of interest because there are few explicit examples in the literature of Tits cones for Kac-Moody groups beyond the affine, hyperbolic, and Lorentzian cases. 

The geometry of the Tits cone plays an important role in Gritsenko and Nikulin's definition of a Lorentzian root system. The bilinear form $( , )$ determines a cone $J$ of timelike vectors satisfying $(\mathbf{s}, \mathbf{s})<0$. One half of this cone, called the past cone $J^-$, intersects the fundamental or antidominant cone  $\bar{C}$. Part of the Lorentzian condition is that $A$ should coincide with $J^-$. More precisely, $\bar{C}$ is a polytope with cusps at infinity given by the fundamental weights, and finite hyperbolic volume as it intersects the hyperbolic space $(\mathbf{s}, \mathbf{s})=-1$ in $J^-$. Its $W$-translates will then cover the hyperbolic space. This occurs in the rank 3 root subsystem $H^{(3)}_{71}$ of $\Phi$.

We will see, however, that none of this holds in $\Phi$. Each fundamental weight $\omega_i$ satisfies $(\omega_i, \omega_i)=\tfrac{1}{8}$ so these vectors are spacelike, not timelike. Thus $\bar{C}$ extends beyond the cone $J^-$, and has infinite hyperbolic volume. The Tits cone $\bar{A}$ includes $J^-$ as a proper subset, and is much more geometrically complicated than $J^-$.  We will see that the intricate structure of Apollonian packings carries over to $A$.

In order to visualize $A$, we must cut down its dimension. First, since $A$ is a tube domain, it suffices to describe the real part. Because the real part of each region $w(C)$ for $w\in W$ is a simplicial cone defined by a system of homogeneous linear inequalities, the real part of $A$ is a cone through the origin in $\R^4$. It is easily shown that the real part of $A$ lies in the half-space $s_1+s_2+s_3+s_4\geq 0$. A cone on the origin in this half-space can be considered as a subset of $\RP^3$. For the rest of this section, by abuse of notation, we will consider $C$, all $w(C)$, and $A$ as subsets of $\RP^3$. The graphics in this section depict the affine section of $A$ along the hyperspace $s_1+s_2+s_3+s_4=1$. For convenience, we will refer to spheres and circles in this section, although the corresponding objects in $\RP^3$ are, more technically, ellipsoids and ellipses.  

Figure \ref{domains}, generated in Mathematica \cite{Mathematica}, show the domains $\bigcup_{\substack{ w \in W \\ \ell(w) \leq \ell}} \, w(C)$ for $\ell=0, 1, 2, 3$ in $\RP^3$. Each $w(C)$ is a 3-simplex.

\begin{figure}[H]
\includegraphics[scale=.18]{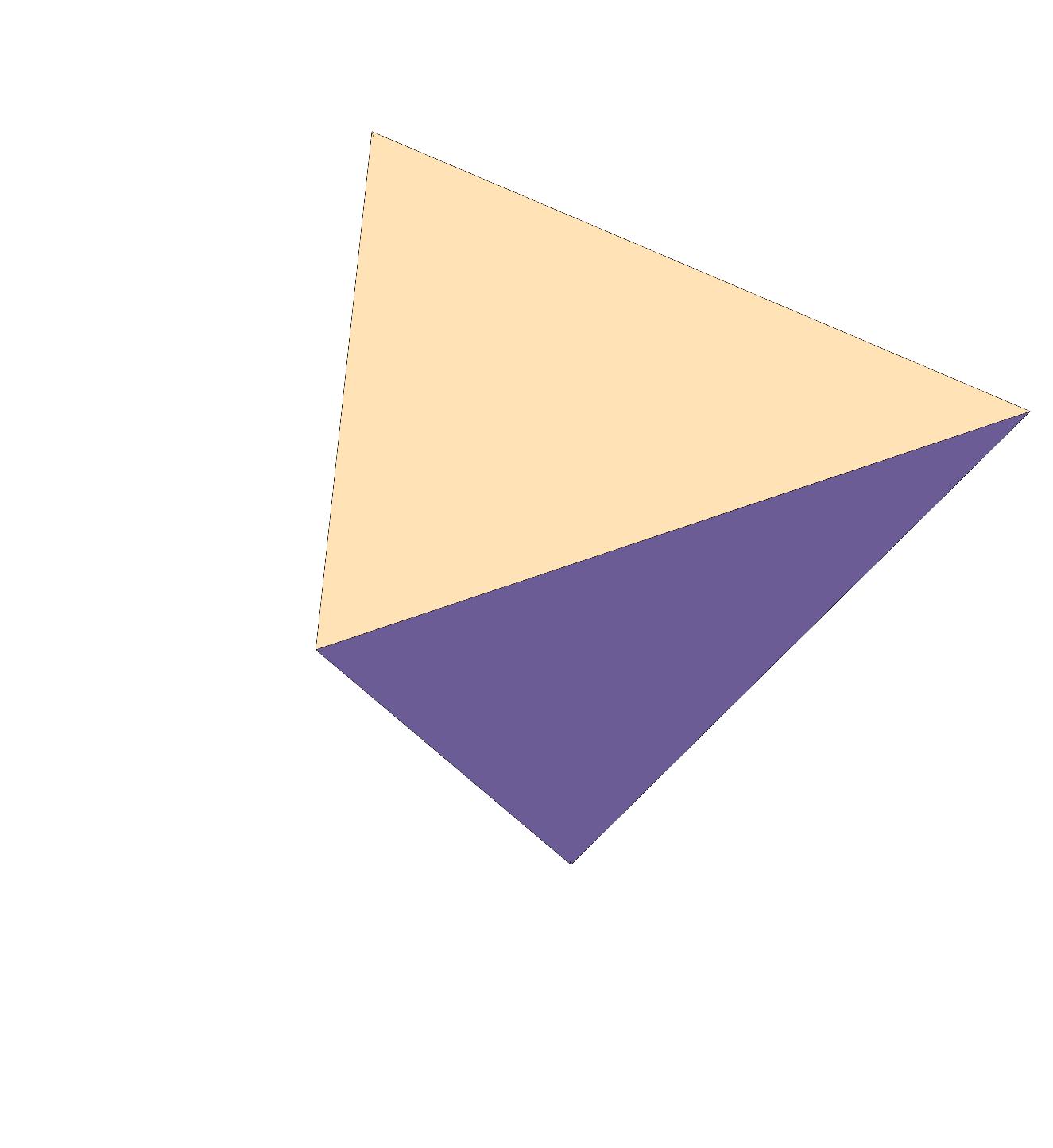}
\includegraphics[scale=.18]{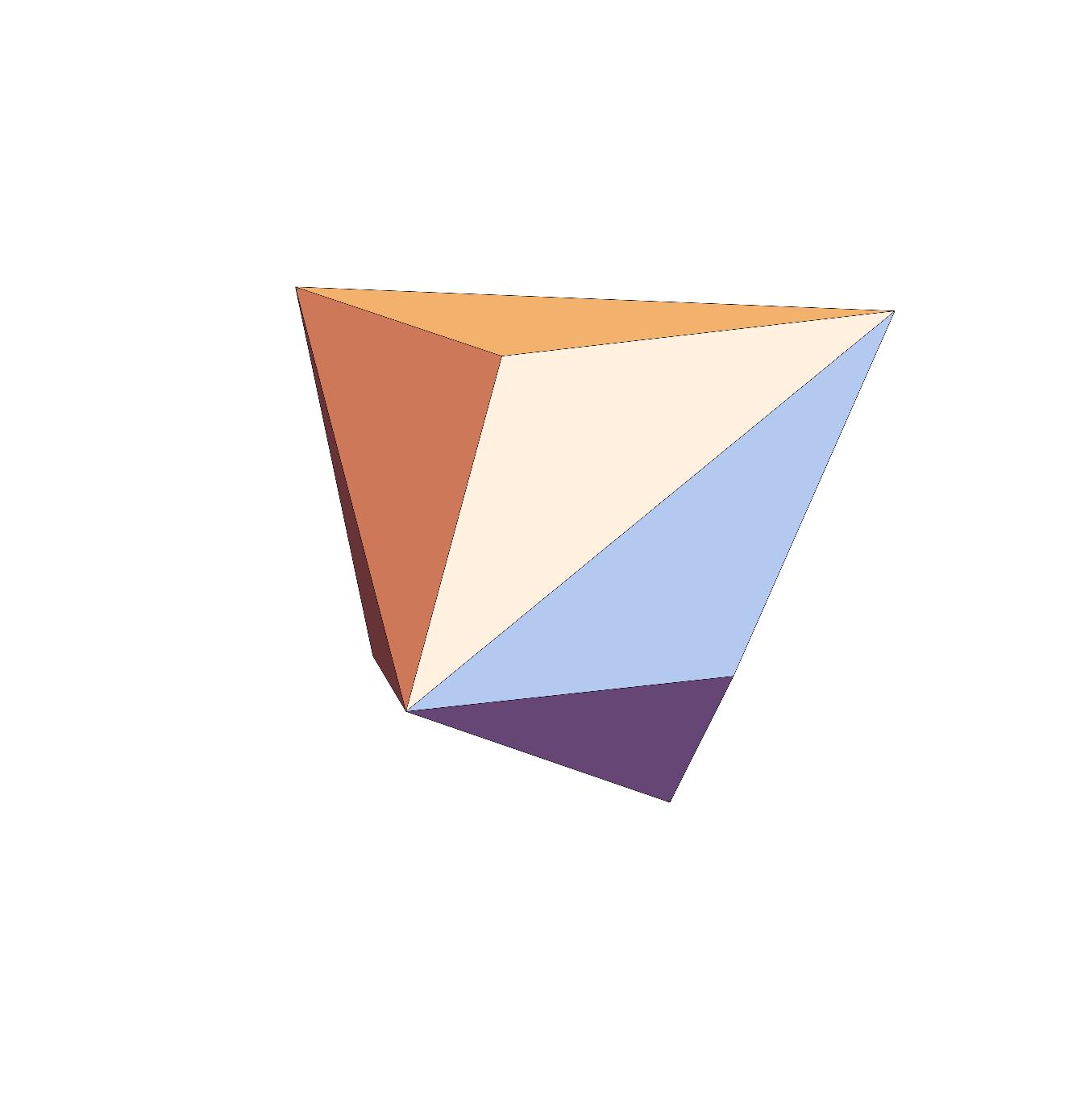}

\includegraphics[scale=.18]{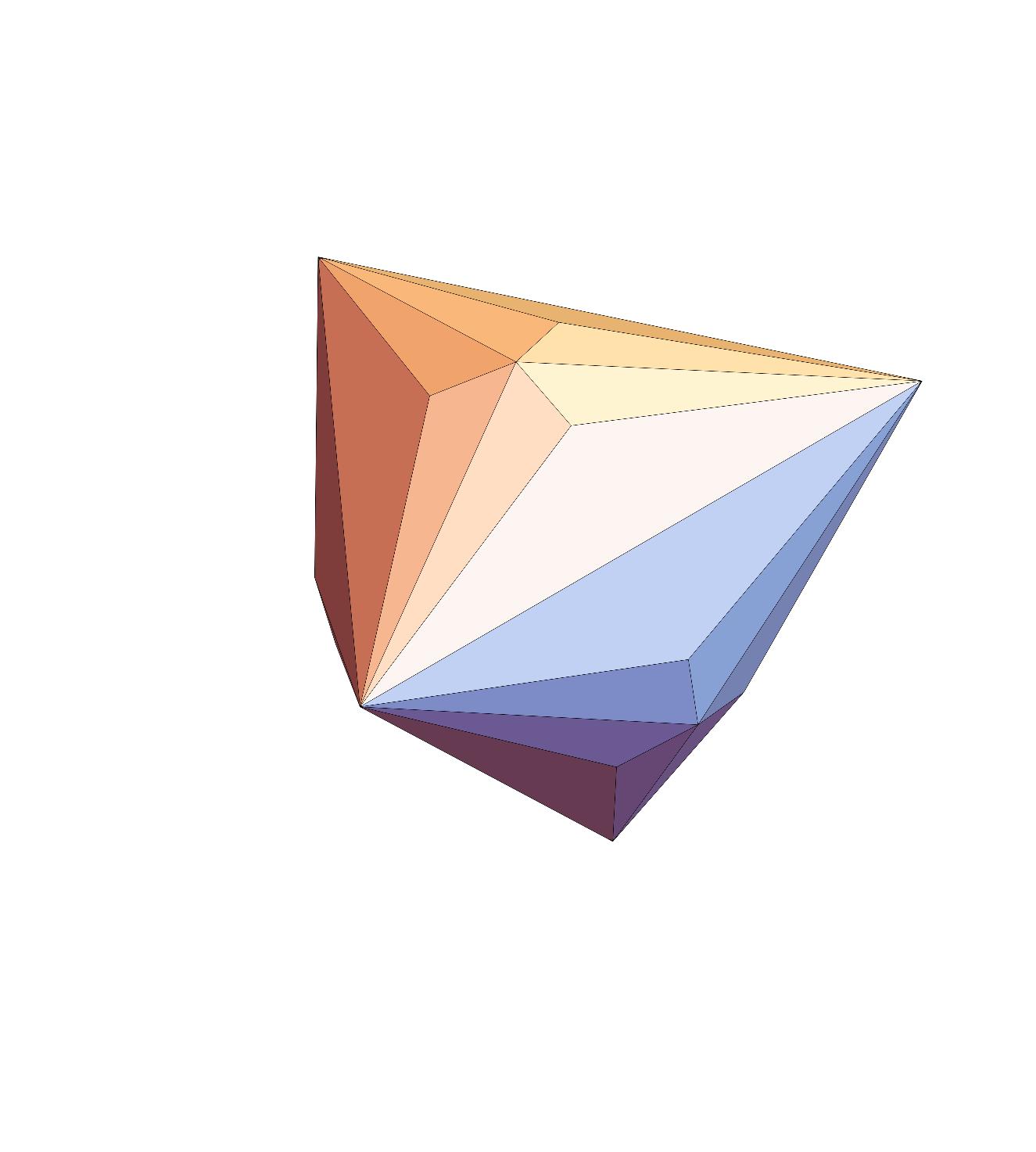}
\includegraphics[scale=.18]{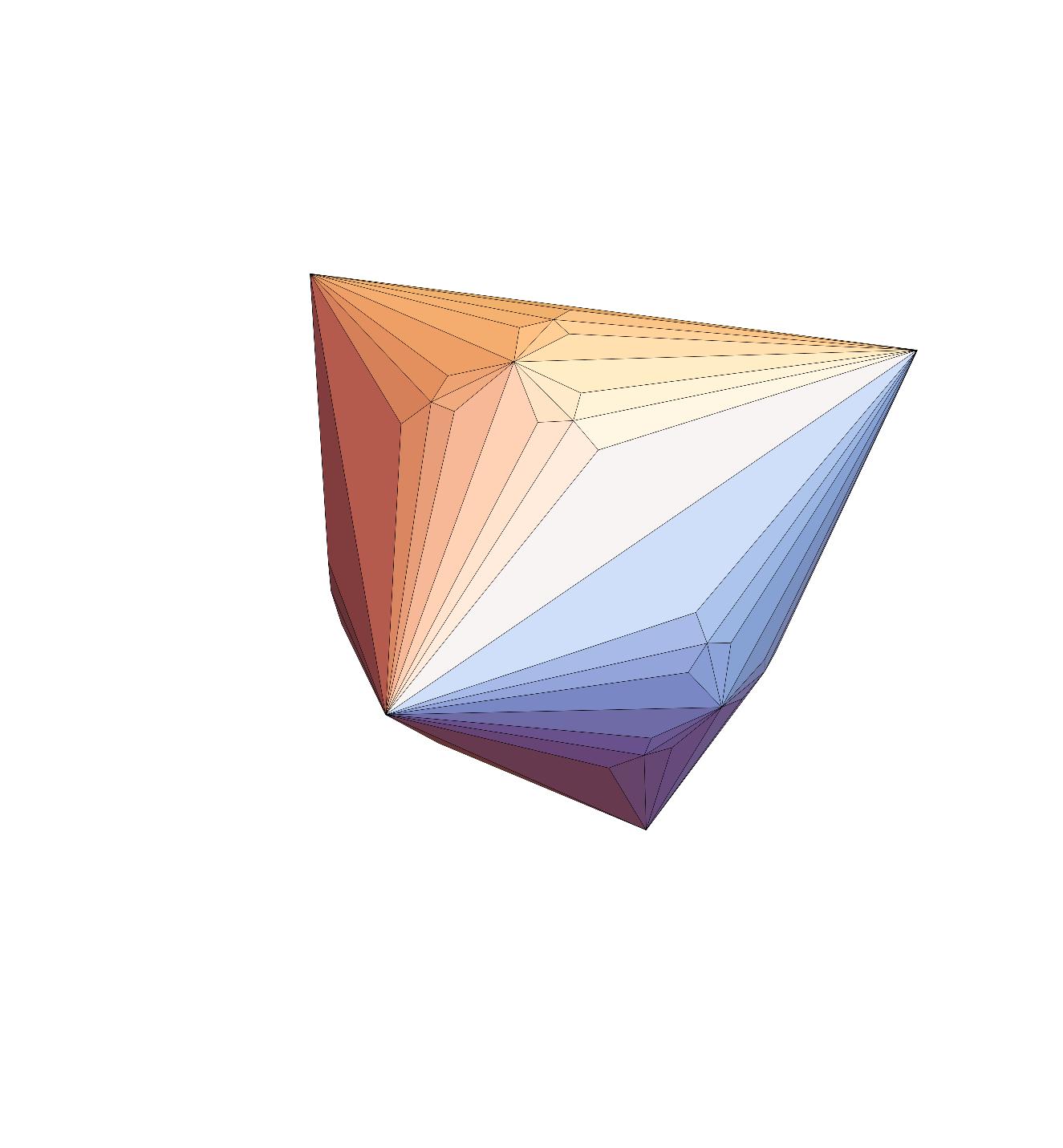}
\caption{Preliminary domains of convergence for $Z(\mathbf{s})$}
\label{domains}
\end{figure}

Define the timelike cone $J= \{ \mathbf{s} \in \RP^3 \, | \, (\mathbf{s}, \mathbf{s})<0\}$, and its boundary, the lightlike cone $N=\{ \mathbf{s} \in \RP^3 \, | \, (\mathbf{s}, \mathbf{s})=0\}$. $N$ is a sphere in projective space, and $J$ is the region enclosed by the sphere. Both are $W$-invariant. 

The 3-simplex $C$ is has vertices $\omega_1$, $\omega_2$, $\omega_3$, $\omega_4$. The six edges are segments along $s_i=s_j=0$ for some $i, j$. Each edge is tangent to $N$. If we fix a vertex $\omega_i$, the three points of tangency on edges through $\omega_i$ lie in the plane $s_1+s_2+s_3+s_4-2s_i=0$. The four planes $s_1+s_2+s_3+s_4-2s_i=0$ intersect $N$ in four mutually tangent circles. 

We use these circles as a base quadruple to generate an Apollonian packing $\mathcal{T}$ on $N$. Then, for each circle $S \in \mathcal{T}$, let $O_S$ be the open spherical cap on $N$, with boundary $S$ (chosen so that all the sets $O_S$ are disjoint). There is a unique point $\mathbf{p}_S$ such that the cone on $S$ with apex $\mathbf{p}_S$ is tangent to $N$. Let $C_S$ be the bounded cone on $O_S$ with apex $\mathbf{p}_S$, and let $C_S'$ be the unbounded cone. An image of the region $J \cup \bigcup_{S \in \mathcal{T}} \, C_S \subset \RP^3$ is shown in Figure \ref{domaininfinity}. 

\begin{figure}[h]
\center{\includegraphics[scale=.27]{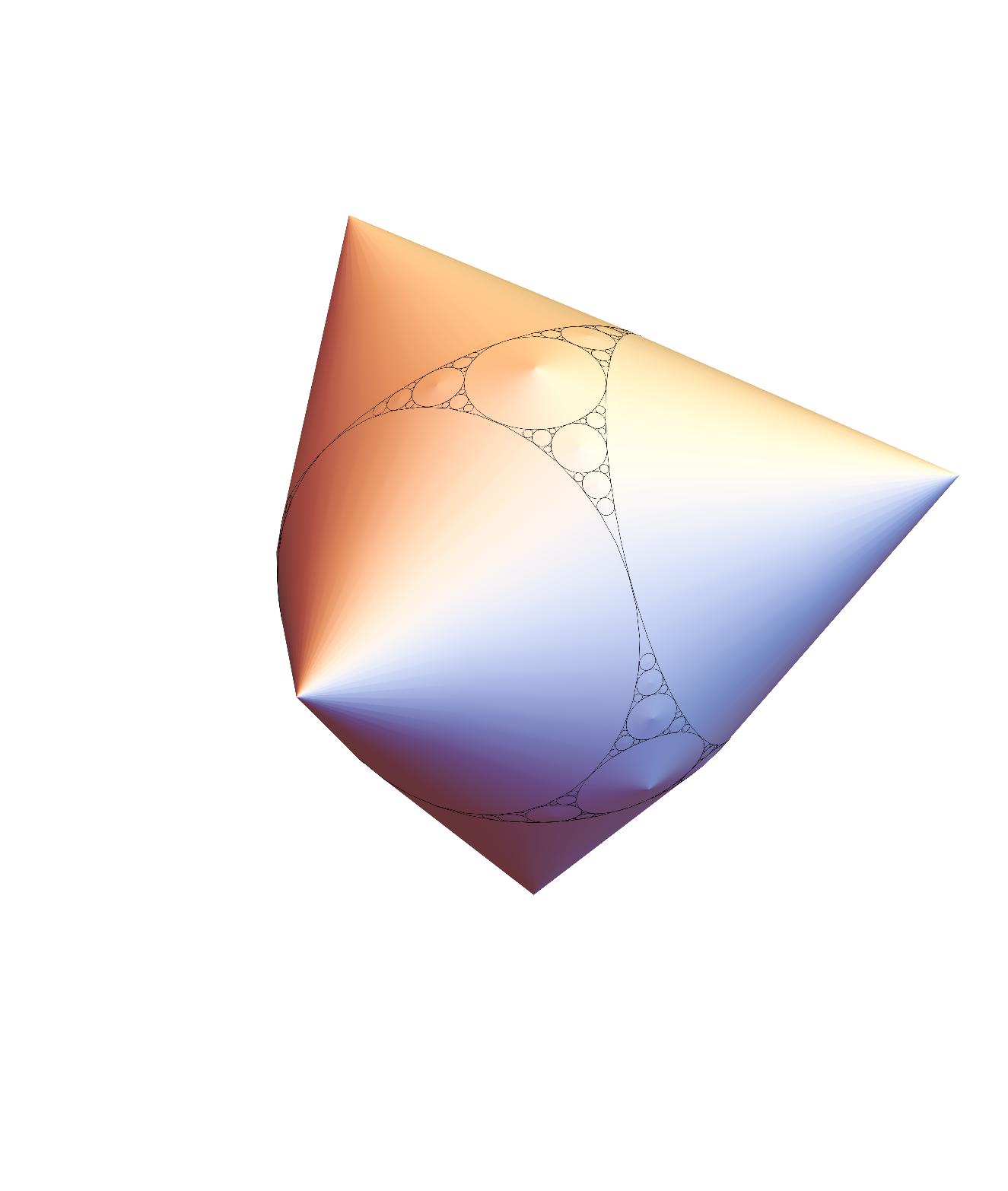}}
\caption{Full domain of convergence for $Z(\mathbf{s})$}
\label{domaininfinity}
\end{figure}

Because the boundary of each cone $C_S$ is tangent to the sphere, the line segments connecting apexes $\mathbf{p}_S$ for tangent circles $S \in \mathcal{T}$ are contained in the boundary of both cones, and tangent to the sphere. Such line segments are dense in the boundary of the region $J \cup \bigcup_{S \in \mathcal{T}} \, C_S$. The line segments connecting apexes of non-adjacent cones must intersect $J$; otherwise the two cones would intersect each other.

Let $F \subset N$ denote the fractal set of points in $N$ not on any circle $S \in \mathcal{T}$ or in any spherical cap $O_S$. 

\begin{proposition} \label{prop:unionintersection}
\begin{equation}
\bigcap_{S \in \mathcal{T}} \, C_S' \, = \, J \cup F \cup \bigcup_{S \in \mathcal{T}} \, C_S
\end{equation}
\end{proposition} 

\begin{proof}
Given two points $\mathbf{p}_1, \mathbf{p}_2 \in \RP^3 \setminus J$, we may form two cones tangent to $N$ with these points as apexes. We will denote the open regions enclosed by these cones as $C_1'$ and $C_2'$, and the circles of tangency as $S_1$ and $S_2$ (For points on $N$, the cone becomes a half-space, and the circle becomes a single point.) If the circles $S_1$ and $S_2$ are externally disjoint, then each point is contained in the interior of the other cone. If the circles are externally tangent, then each point is on the boundary of the other cone. If the circles intersect non-tangentially, then each point is outside the other cone.  If the circles are internally tangent, then the inner point is on the boundary of the outer cone. If one circle is internal to the other, then the inner point is enclosed by the outer cone. 

First we will show that the union on the right side is contained in the intersection on the left. The ball $J$ is contained in every region $C_S'$. The set $F$ is also contained in every $C_S'$ because it does not intersect the boundary of any of these sets. Finally, a pair of circles in the packing $\mathcal{T}$ must be externally disjoint or externally tangent. It follows that each point $\mathbf{p}_S$ is contained in all the regions $\bar{C_S'}$, and therefore that every bounded region $C_S$ is contained in all the unbounded regions $C_S'$. 

Next we show the opposite inclusion. For a point $\mathbf{s}$ contained in every region $C_S'$, either $\mathbf{s} \in J$ or we may form the cone tangent to $N$ with $\mathbf{s}$ at its apex. The circle of tangency cannot intersect any circle $S \in \mathcal{T}$. This is possible only if the circle is internal to some $S$, in which case $\mathbf{s} \in C_S$, or if the circle degenerates to a single point in $F$.
\end{proof}

We now come to the main result of this section.

\begin{theorem} \label{theorem:Ageometry} We have
\begin{equation}
A=J \cup \bigcup_{S \in \mathcal{T}} \, C_S
\end{equation}
in $\RP^3$. Moreover, each vertex of each simplex $\bar{w(C)}$ is $\mathbf{p}_S$ for some $S\in \mathcal{T}$, and each edge of each $\bar{w(C)}$ is the line segment connecting the apexes of two adjacent cones in the packing. 
\end{theorem}

\begin{proof} 
The four points $\omega_i$ are the apexes of four cones tangent to $N$, which intersect $N$ in the base quadruple of $\mathcal{T}$. The four reflections $\sigma_i$, considered as transformations of $\RP^3$, map $N$ to itself. Because each circle on $N$ is contained in a unique plane, and $\sigma_i$ maps planes to planes, $\sigma_i$ must map circles to circles on $N$. Each $\sigma_i$ fixes the plane $s_i=0$. For $j \neq i$, $\sigma_i$ maps the plane $s_1+s_2+s_3+s_4-2s_j=0$ to itself. Thus $\sigma_i$ maps three of the four circles in the base quadruple to themselves, and moves the fourth circle. Because $\sigma_i$ must preserve tangency, it maps the packing $\mathcal{T}$ to itself. From the action on the base quadruple, we can see that $W$ acts on $\mathcal{T}$ in the standard way.

Because the transformations $\sigma_i$ map the packing $\mathcal{T}$ to itself, and because they preserve tangency, they must map the collection of cones $C_S$ for $S \in \mathcal{T}$ to itself. In particular, the set of apexes $\mathbf{p}_S$ for $S \in \mathcal{T}$ is the orbit of $W$ on the four initial points $\omega_1$, $\omega_2$, $\omega_3$, $\omega_4$. This set is the zero-skeleton of $A$. Each $w(\bar{C})$ is a simplex with vertices at the four points $\mathbf{p}_S$ corresponding to a Descartes quadruple of circles $S\in \mathcal{T}$. The edges of this simplex are line segments, tangent to $N$, connecting apexes of adjacent cones--these form the one-skeleton of $A$. The faces and interior of the simplex lie in the union of $J$ with the four cones of the Descartes quadruple. 

It remains to show that $A \supseteq J \cup \bigcup_{S \in \mathcal{T}} \, C_S$. We have remarked that the boundary of $A$, specifically the 2-skeleton, is dense in the boundary of each cone $C_S$. Since $A$ is convex, it follows that $C_S$ is contained in $A$. Since the cones $C_S$ intersect the boundary of $J$ in another dense set, it follows that $J$ is contained in $A$.
\end{proof}

Note that even conditional convergence is impossible along the 2-skeleton of $A$. Conditional convergence may be possible at other points of $\partial A$. The question of conditional convergence will not be discussed further here. 

Recall that the set $B$ from the proof of Theorem \ref{theorem:A} consists of points $\mathbf{s}$ such that every $W$-translate $w(\mathbf{s})$ has three positive coordinates and one negative. We can classify a point $\mathbf{s}\in \RP^3$, assuming $s_1+s_2+s_3+s_4\geq 0$, based on the the number of negative coordinates of its $W$-translates.

\begin{proposition}
Any point of $\RP^3 \setminus \bar{A}$ has a $W$-translate with three negative coordinates, or two negative coordinates and one zero coordinate.
\end{proposition}
\begin{proof}
Given a point $\mathbf{s} \in \RP^3 \setminus \bar{A}$, we can form the cone tangent to $N$ with apex $\mathbf{s}$. As in the proof of \ref{prop:unionintersection}, the circle of tangency cannot be internal to any circle in $\mathcal{T}$. It must contain some point of the boundary of some circle of $\mathcal{T}$ internally, and then, because it contains a neighborhood of that point, it must in fact contain a circle $S \in \mathcal{T}$ internally. It follows that the point $\mathbf{s}$ must be contained in the region bounded by the cone opposite to $C_S'$, i.e. the cone with apex $\mathbf{p}_S$ and half-lines in the opposite direction of those in $C_S'$. 

As demonstrated in Theorem \ref{theorem:Ageometry}, the region $C_S$ is contained in a union of simplices $w(C)$ for those $w \in W$ which map the base quadruple to a quadruple involving $S$. Each of these simplices is defined by four inequalities. If we drop the inequality which is strictly satisfied by $\mathbf{p}_S$, we obtain an unbounded triangular cone with apex $\mathbf{p}_S$. The region $C_S'$ is the union of the regions enclosed by these unbounded triangular cones. A point in the cone opposite $C_S'$ must lie in the opposite of one of the these triangular cones. Thus, it satisfies three inequalities opposite to those which define the triangular cone. If it lies opposite the interior of a triangular cone, it satisfies the three inequalities strictly. 

By applying a $W$-translation which maps $S$ to a circle of the base quadruple, we find that some $w(\mathbf{s})$ has three nonpositive coordinates, with at most one coordinate equal to zero.
\end{proof}

The following proposition describes the region $B$.

\begin{proposition}
Let $\mathbf{s}$ be a point of $\partial A$ which is not in the 2-skeleton of $A$. Then every $W$-translate of $\mathbf{s}$ has exactly three positive coordinates and one negative. Thus $B$ is the boundary $\partial A$ with the two-skeleton of $A$ removed.
\end{proposition}
\begin{proof}
It suffices to show that $\mathbf{s}$ satisfies exactly three of the four inequalities which define each simplex $w(C)$ for $w \in W$. Then we see that the translate of  $\mathbf{s}$ by every $w^{-1}$ has three positive coordinates and one negative. The simplex $w(C)$ has vertices $\mathbf{p}_{S_1}$, $\mathbf{p}_{S_2}$, $\mathbf{p}_{S_3}$, $\mathbf{p}_{S_4}$ for a Descartes quadruple of circles $S_1, S_2, S_3, S_4 \in \mathcal{T}$. There are two cases: if $\mathbf{s}$ does not lie on $\partial C_{S_i}$ for $i=1, 2, 3, 4$, then since $\mathbf{s} \in \partial A$, it must lie on the surface in the interstitial region between three of the circles $S_i$. Assume that $\mathbf{s}$ lies between $S_1$, $S_2$, and $S_3$. Then $\mathbf{s}$ is on the right side of the plane through $\mathbf{p}_{S_1}$, $\mathbf{p}_{S_2}$, $\mathbf{p}_{S_4}$, the plane through $\mathbf{p}_{S_1}$, $\mathbf{p}_{S_3}$, $\mathbf{p}_{S_4}$, and the plane through $\mathbf{p}_{S_2}$, $\mathbf{p}_{S_3}$, $\mathbf{p}_{S_4}$, but it is on the wrong side of the plane through $\mathbf{p}_{S_1}$, $\mathbf{p}_{S_2}$, $\mathbf{p}_{S_3}$. 

For the other case, assume that $\mathbf{s}$ lies on $\partial C_{S_1}$. Then there are three line segments connecting  $\mathbf{p}_{S_1}$ to $\mathbf{p}_{S_2}$, $\mathbf{p}_{S_3}$, and $\mathbf{p}_{S_4}$. The point $\mathbf{s}$ must lie between two of these segments. Assume that it lies between the segment to $\mathbf{p}_{S_2}$ and the segment to $\mathbf{p}_{S_3}$. Then $\mathbf{s}$ is on the right side of the plane through $\mathbf{p}_{S_1}$, $\mathbf{p}_{S_2}$, $\mathbf{p}_{S_4}$, the plane through $\mathbf{p}_{S_1}$, $\mathbf{p}_{S_3}$, $\mathbf{p}_{S_4}$, and the plane through $\mathbf{p}_{S_2}$, $\mathbf{p}_{S_3}$, $\mathbf{p}_{S_4}$, but it is on the wrong side of the plane through $\mathbf{p}_{S_1}$, $\mathbf{p}_{S_2}$, $\mathbf{p}_{S_3}$.
\end{proof}

The results of this section allow us to rediscover the geometry of Apollonian packings with just the Descartes quadratic form as a starting point. This raises the natural question: what generalizations of Apollonian packings can we obtain by starting from a different quadratic form?

\bibliography{Apollonianbib}
\bibliographystyle{amsplain}

\end{document}